\documentclass[reqno]{amsart}

\usepackage{graphicx,subfigure,color}

\numberwithin{equation}{section}

\usepackage[latin1]{inputenc}
\usepackage[english]{babel}

\usepackage{amsmath,amsthm,amsfonts,latexsym,amssymb}
\usepackage[colorlinks]{hyperref}
\hypersetup{linkcolor=blue,citecolor=blue,filecolor=black,urlcolor=blue}
\usepackage{comment}

\usepackage{color}

{ \theoremstyle{plain}
\newtheorem{theorem}{Theorem}[section]
\newtheorem{proposition}[theorem]{Proposition}
\newtheorem{lemma}[theorem]{Lemma}

  \theoremstyle{remark}
\newtheorem{remark}[theorem]{Remark}
  \theoremstyle{definition}
\newtheorem{definition}[theorem]{Definition}
}

\def\R{\mathbb{R}}

\def\N{\mathbb{N}}

\title[Multiple solutions for $(p,q)$-Laplacian problems]{Multiple solutions for asymptotically $q$-linear $(p,q)$-Laplacian problems}

\author{Francesca Colasuonno}

\address{Dipartimento di Matematica \newline\indent
Alma Mater Studiorum Universit\`a di Bologna
\newline\indent
Piazza di Porta San Donato 5 - 40126 Bologna, Italy}
\email{francesca.colasuonno@unibo.it}

\begin{document}
	



\begin{abstract}
We investigate the existence and the multiplicity of solutions of the problem
$$
\begin{cases}
-\Delta_p u-\Delta_q u = g(x, u)\quad & \mbox{in }  \Omega,\\
\displaystyle{u=0}  & \mbox{on }  \partial\Omega,
\end{cases}
$$
where $\Omega$ is a smooth, bounded domain of $\R^N$, $1<p<q<\infty$, and the nonlinearity $g$ behaves as $u^{q-1}$ at infinity. 
We use variational methods and find multiple solutions as minimax critical points of the associated energy functional. Under suitable assumptions on the nonlinearity, we cover also the resonant case.
\end{abstract}

\maketitle

\section{Introduction}
We consider the following Dirichlet $(p,q)$-Laplacian problem
\begin{equation}\label{eq:main}
\begin{cases}
-\Delta_p u-\Delta_q u = g(x, u)\quad & \mbox{in }  \Omega,\\
\displaystyle{u=0}  & \mbox{on }  \partial\Omega,
\end{cases}
\end{equation}
where $\Omega$ is a bounded domain of $\R^N$ ($N\geq 2$) with Lipschitz boundary $\partial\Omega$, $1<p<q<\infty$, and the nonlinearity $g(x,\cdot)$ is asymptotically $q$-linear, meaning that $g(x,t)\sim \ell_\infty |t|^{q-2}t$ as $|t|\to\infty$ for some constant $\ell_\infty$. 

Some of the difficulties arising in the study of this problem come from the non-homogeneity of the operator $-\Delta_p-\Delta_q$. The interest in $(p,q)$-Laplacians and, in general, in non-homogeneous operators has considerably increased since the seminal papers \cite{Marcellini, Marcellini2} by Marcellini, in late eighties, on the regularity of minimizers of the so-called functionals with {\it non-standard growth}. In this setting, both $(p,q)$ and $p(x)$-growth conditions have been widely cosidered, see \cite{DHHR} and the references therein for a comprehesive monograph on the Lebesgue and Sobolev spaces involved in the $p(x)$ variable exponent case. More recently, in \cite{BCM,ColMin}, many others progresses were achieved in the study of energy functionals related to the operators $-\mathrm{div}(|\nabla u|^{p-2}\nabla u+a(x)|\nabla u|^{q-2}\nabla u)$. In these papers, the weight function $a(x)\ge 0$ switches two different elliptic behaviors, justifying the name of {\it double phase} functionals. These functionals were first introduced by Zhikov in \cite{Z} to provide models for strongly anisotropic materials. In that setting, the exponents $p$ and $q$ cannot be too far from each other, the bound $q/p<(1+\alpha/N)$ for some $\alpha\in(0,1]$ is needed both to develop a regularity theory and to prove some classical inequalities like a Poincar\'e-type one, in suitable Orlicz-Sobolev spaces, see \cite[Remark 2.19]{CSq2}. 
On the other hand, $(2,q)$-Laplacian operators naturally arise in the Born-Infeld theory of nonlinear electromagnetism, where the leading operator, i.e., the Minkowski-curvature operator $-\mathrm{div}\left(\frac{\nabla u}{\sqrt{1-|\nabla u|^2}}\right)$, can be approximated by a truncated series of $2h$-Laplacians, $h\in\N$, cf. \cite{DFJ,BDP,FOP}. In this case, the highest exponent in the truncated series should morally go to infinity, and so no bounds on $q/2$ are admissible.
In the same spirit, in the problem under consideration, the exponents $p$ and $q$ can be arbitrarily far. We observe that all the arguments in this paper can be adapted to operators of the form $-\mathrm{div}(|\nabla u|^{p-2}\nabla u+a(x)|\nabla u|^{q-2}\nabla u)$, with the weight function $a(x)$ satisfying $a\in C^1(\bar\Omega)$, $a>0$ (it is enough to endow the functional space $W^{1,q}_0(\Omega)$ with the equivalent norm $\|a(x)\nabla u\|_q$). On the contrary, things change a lot for the more general case in which $a(x)$ can vanish somewhere, this case should be treated in suitable Orlicz-Sobolev spaces, the above mentioned bound on $q/p$ should be required, and, in particular, being $p(1+1/N)<p^*$, the exponent $q$ should be taken $p$-subcritical. While in our setting, since $a\equiv 1$ never vanishes, we are allowed not to require any relation between $q$ and $p$ (note that $q>p$ is not an assumption, since the roles of $p$ and $q$ are interchangeable). 
\smallskip 

Let us now introduce in details the hypotheses required on the nonlinearity of the problem. We assume that there exist $\ell_\infty\in \R$ and $f:\Omega\times\R\rightarrow \R$ such that
\begin{equation}\label{eq:nonlinearity}
g(x,t) = \ell_\infty |t|^{q-2}t + f(x,t),
\end{equation}
so that the problem can be written as
\begin{equation}\label{eq:main-f}
\begin{cases}
-\Delta_p u-\Delta_q u = \ell_\infty |u|^{q-2}u + f(x,u)\quad & \mbox{in }  \Omega,\\
\displaystyle{u=0}  & \mbox{on }  \partial\Omega.
\end{cases}
\end{equation}

In what follows, we denote by $\sigma(-\Delta_q)$ the spectrum of the Dirichlet $q$-Laplacian operator, namely the set of $\lambda's$ in $\mathbb R$ for which the problem
$$
\begin{cases}
-\Delta_q u=\lambda |u|^{q-2}u \quad&\mbox{in }\Omega,\\
u=0&\mbox{on }\partial\Omega
\end{cases}
$$
has a nontrivial weak solution. 
The nonlinearity $f$ satisfies the following assumptions:
\begin{itemize}
\item[$(f)$] 
$f$ is a Carath\'eodory function (i.e., $f(\cdot,t)$ is measurable in $\Omega$
 for all $t\in\R$ and $f(x,\cdot)$ is continuous in $\R$ for a.e. $x\in \Omega$) 
 and $\max_{|t|\leq R}|f(\cdot,t)|\in L^\infty(\Omega)$ for all $R>0$;
\item[$(f_\infty)$] $\lim_{|t|\to \infty}\frac{f(x,t)}{|t|^{q-2}t} = 0$ uniformly in $x \in \Omega$;
\item[$(f_\mathrm{cpt})$] one of the two assumptions holds: 
\begin{itemize}
\item[$(f_\mathrm{nr})$] $\ell_\infty\not\in\sigma(-\Delta_q)$;
\item[$(f_\mathrm{r})$] $\ell_\infty\in\sigma(-\Delta_q)$ and $\lim_{|t|\to \infty}(f(x,t)t-qF(x,t)) = +\infty$ uniformly in $x \in \Omega$, where $F(x,t):=\int_0^t f(x,s)ds$; 
\end{itemize}
\item[$(f_0)$] one of the two assumptions holds:
\begin{itemize}
\item[$(f_0^-)$] $\lim_{t\to 0}\frac{f(x,t)}{|t|^{q-2}t} =: \ell_0\in [-\infty,0]$ uniformly in $x \in \Omega$;
\item[$(f_0^+)$] $\lim_{t\to 0}\frac{f(x,t)}{|t|^{p-2}t} =: \ell'_0\in (0,\infty]$ uniformly in $x \in \Omega$;
\end{itemize}
\item[$(f_\mathrm{sym})$] $f(x,\cdot)$ is odd for a.e. $x\in\Omega$.
\end{itemize}
\smallskip 

Problem \eqref{eq:main-f} has a variational structure, so the solutions are found as critical points of the associated energy functional $I:W^{1,q}_0(\Omega)\to \R$ defined as
$$
I(u)=\frac{1}{p}\int_\Omega|\nabla u|^p dx+\frac{1}{q}\int_\Omega|\nabla u|^q dx-\frac{\ell_\infty}{q}\int_\Omega |u|^q dx -\int_\Omega F(x,u)dx,
$$
cf. Subsection \ref{subsec:varset}. 
We observe that, beyond the non-homogeneity of the operator, another feature of the problem that makes the analysis more interesting is the behavior of the nonlinearity at infinity. Indeed, the asymptotic $q$-linearity of $g(x,\cdot)$ in particular implies that the Ambrosetti-Rabinowitz condition --which is responsible for Palais-Smale sequences to be bounded-- is not satisfied. This prevents the use of the classical Mountain Pass or Symmetric Mountain Pass Theorem for the existence of solutions to problem \eqref{eq:main}, see the Introduction of \cite{LZ} for interesting comments on the topic. For existence results in the case of $q$-superlinear and subcritical nonlinearities in the $(p,q)$-setting, we refer to \cite{PS} in bounded domains, and to \cite{BCS2} in the whole space $\R^N$.

Due to the behavior of $g(x,\cdot)$ at infinity, one can expect some interaction with the spectrum of the $q$-Laplacian, cf. \cite{BCS,LZ}. In particular, when $\ell_\infty$ is an eigenvalue of $-\Delta_q$, a stronger assumption on $f$ is needed to get compactness, see $(f_\mathrm{r})$. Furthermore, due to the symmetry condition $(f_\mathrm{sym})$, the energy functional $I$ is even, so that if $u$ is a critical point of $I$ at some critical value $c$, also $-u$ has the same property. This is the reason why in the statement of the main theorem below, we always refer to {\it pairs} of solutions. 
We believe that even without the symmetry condition $(f_\mathrm{sym})$, it is possible to obtain existence results by applying a Linking Theorem as in \cite{BCS}.

In Section \ref{sec:sec2}, we will introduce the definitions of two suitable sequences $(\eta^{(\alpha)}_h), (\nu^{(\alpha)}_h)$ of quasi-eigenvalues related to the $(p,q)$-Laplacian operator, 
with $\alpha=0$ or 1.
For the exact definitions of $(\eta^{(\alpha)}_h), (\nu^{(\alpha)}_h)$, we refer to \eqref{eq:etah} and \eqref{eq:nuh-def}, respectively.
These sequences play an important role in the proof of the geometry conditions for the functional $I$. 

We are now ready to state our main result.

\begin{theorem}\label{thm:main}
Assume that  $(f)$, $(f_\infty)$, $(f_\mathrm{cpt})$, and $(f_\mathrm{sym})$ hold. Suppose further that one of the following assumptions is fulfilled:
\begin{itemize}
\item[$(H_-)$:] $(f_0^-)$ holds and there exist $h,k\in\N$, with $k\ge h$, such that $\ell_0+\ell_\infty<\eta^{(0)}_h$ and $\ell_\infty>\nu^{(0)}_k$;
\item[$(H_+)$:] $(f_0^+)$ holds and there exist $h,k\in\N$, with $k\ge h$, such that $\ell_\infty<\eta_h^{(0)}$ and $\ell'_0>\nu_k^{(1)}$.
\end{itemize}
Then, problem \eqref{eq:main-f} has at least $k-h+1$ distinct pairs of non-trivial solutions.
\end{theorem}

All the solutions found are minimax critical points of the associated energy functional.
The proof of our main theorem is based on an abstract result proved in \cite{BBF,BCS} by using the pseudo-index theory  related to the Krasnosel'ski\u{\i} genus, see Theorem~\ref{thm:multiplicity}. In order to apply this result, we prove that the energy functional $I$ satisfies a compactness assumption, the Cerami's weaker variant of the Palais-Smale condition, both in the non-resonant case and in the resonant one under $(f_\mathrm{r})$, cf. Lemma \ref{lem:CPS}. In the proof of this lemma, neither the symmetry nor the behavior near zero of $f$ enter at all. Moreover, as usual, in order to find multiple minimax critical points, it is also needed to show that the energy functional $I$ has the right geometry. This part of the proof involves the behavior of the nonlinearity both at infinity and at zero, and is responsible for the assumptions on $\ell_\infty$, $\ell_0$, and $\ell'_0$ in $(H_-)$ and $(H_+)$.

Some remarks on the statement are now in order. In both cases $(H_-)$ and $(H_+)$, the larger $|\ell_0|$, $\ell'_0$, the higher the possibility of finding solutions. In fact, in the limit cases $|\ell_0|=\ell'_0=\infty$, we get the highest number of solutions. If $\ell_0=-\infty$, certainly $\ell_0+\ell_\infty<\eta^{(0)}_1$, hence \eqref{eq:main-f} admits at least $k$ pairs of distinct solutions. 
If $\ell'_0=\infty$, certainly $\ell'_0>\nu^{(1)}_k$ is satisfied for every $k\in\N$. Moreover, being $(\eta_h^{(0)})$ divergent (see Proposition \ref{prop:etahdiv}), there always exists $\bar h\in \N$ for which $\ell_\infty<\eta_{\bar h}^{(0)}$. Therefore, in this case \eqref{eq:main-f} admits infinitely many pairs of distinct solutions. 

On the other hand, it is not clear whether the case $\ell_0=0$ can be covered, see Remarks \ref{rmk:remark-eta1} and \ref{rmk:lastremark}. The cases $\ell_0=-\infty$ and $\ell'_0<\infty$ include the $p$-behavior at zero. We refer to \cite{BT,BT2,MM}, for problems with $f\sim |t|^{p-2}t$ in a neighborhood of 0, and the so-called two-parameter eigenvalue problem for the $(p,q)$-Laplacian.

Finally, let $\lambda_1(q)$ be the first eigenvalue of the Dirichlet $q$-Laplacian. Being $\nu^{(0)}_k\ge\nu^{(0)}_1=\lambda_1(q)$ (see Subsection \ref{subsubsec:quasieig}), condition $(H_-)$ is never satisfied if $\ell_\infty \le \lambda_1(q)$. Similarly, since $\nu_1^{(1)}\ge \lambda_1(p)$, with $\lambda_1(p)$ the first eigenvalue of the Dirichlet $q$-Laplacian, condition $(H_+)$ is never satisfied if $\ell'_0 \le \lambda_1(p)$, this is coherent with the non-existence result in \cite[Proposition 1]{BT}.

The paper is organized as follows. In Section \ref{sec:sec2}, we present the abstract result that we will apply, and introduce the variational setting and the two sequences of quasi-eigenvalues $(\eta_h^{(\alpha)})$ and $(\nu_h^{(\alpha)})$. In Section \ref{sec:sec3}, we prove the multiplicity result through the intermediate steps of showing that the energy functional satisfies the compactness condition and has the right geometry.

\section{Preliminary results}
\label{sec:sec2}

\subsection{Abstract results}
\begin{definition}
Let $X$ be a Banach space. A $C^1$-functional $I:X\to\R$ satisfies the Cerami-Palais-Smale condition ({\it (CPS)-condition} for short) if every sequence $(u_n)\subset X$ such that 
$$
I(u_n)\to c\in \R\quad\mbox{and}\quad \|I'(u_n)\|_{X'}(1+\|u_n\|_X)\to 0 \mbox{ as }n\to\infty
$$ 
admits a convergent subsequence. 
\end{definition}
 
We will apply the following multiplicity result, see \cite[Theorem 2.9]{BBF} for a proof in Hilbert spaces and \cite[Theorems 2.6 and 2.7]{BCS} for Banach spaces. In particular, for the version that appears as {\it (resp. ...)} in the statement, we refer to \cite[Remark~2.8]{BCS}. 

\begin{theorem}\label{thm:multiplicity}
Let $X$ be a Banach space and for $\rho>0$ denote $S_\rho:=\{u\in X\,:\,\|u\|_X=\rho\}$. Suppose that the functional $I\in C^1(X,\R)$ satisfies the following properties 
\begin{list}{(\roman{enumi})}{\usecounter{enumi}\labelwidth  5em
\itemsep 0pt \parsep 0pt}
\item $I$ is even;
\item $I$ satisfies $(CPS)$ in $(0,\infty)$, and $I(0)\ge 0$;
\item there exist two closed subspaces  $V,\,W\subset X$ such that $\mathrm{dim}V<\infty$ and $\mathrm{codim}W<\infty$, and two constants $c_\infty>c_0> I(0)$ for which the following assumptions hold 
\begin{itemize}
\item[\em{(a)}] $I(u)\ge c_0$ for every $u\in S_\rho\cap W$ (resp. for every $u\in S_\rho\cap V$);
\item[\em{(b)}] $I(u)\le c_\infty$ for every $u\in V$ (resp. for every $u\in W$).
\end{itemize}
\end{list}
If furthermore $\mathrm{dim} V>\mathrm{codim} W$, then $I$ possesses at least $m=\mathrm{dim}V -\mathrm{codim}W$ distinct pairs of critical points, whose corresponding critical values belong to $[c_0,c_\infty]$.
\end{theorem}

\subsection{Variational Setting}\label{subsec:varset}
Throughout the paper, for $1\le r\le\infty$, we denote with $\|\cdot\|_r$, the usual norm in the Lebesgue space $L^r(\Omega)$. 
We look for solutions of \eqref{eq:main-f} in the Sobolev space $W^{1,q}_0(\Omega)$ endowed with the equivalent norm 
$$
\|u\|:=\|\nabla u\|_q.
$$ 

\begin{definition}
A function $u\in W^{1,q}_0(\Omega)$ is a weak solution of \eqref{eq:main-f} if for every $\varphi\in W^{1,q}_0(\Omega)$ the following distributional identity holds 
$$
\int_\Omega (|\nabla u|^{p-2}+|\nabla u|^{q-2})\nabla u\nabla \varphi dx=\ell_\infty\int_\Omega |u|^{q-2}u\varphi dx +\int_\Omega f(x,u)\varphi dx.
$$
\end{definition}
We observe that, due to the boundedness of $\Omega$, all the integrals above are finite. Indeed, for every  $r<q$, by H\"older's inequality we have 
\begin{equation}\label{eq:embed-Wq-in-Wr}
\int_\Omega|v|^r dx\le |\Omega|^{\frac{q-r}{q}}\|v\|_q^r\quad\mbox{for every }v\in L^q(\Omega).
\end{equation}

The problem \eqref{eq:main-f} has a variational structure, its associated energy functional 
$I:W^{1,q}_0(\Omega)\to\R$ is defined as follows for every $u\in W^{1,q}_0(\Omega)$
$$
I(u)=\frac{1}{p}\int_\Omega|\nabla u|^p dx+\frac{1}{q}\int_\Omega|\nabla u|^q dx-\frac{\ell_\infty}{q}\int_\Omega |u|^q dx -\int_\Omega F(x,u)dx.
$$
It is straightforward to verify that $I$ is of class $C^1$ and that $u\in W^{1,q}_0(\Omega)$ is a weak solution of \eqref{eq:main-f} if and only if it is a critical point of $I$.
 
Moreover, for future use, we introduce the Sobolev critical exponent for the embedding $W^{1,q}(\Omega)\hookrightarrow L^r(\Omega)$ to be 
$$
q^*:=
\begin{cases}
\frac{Nq}{N-q}\quad&\mbox{if } N>q,\\
+\infty&\mbox{if }N\le q
\end{cases} 
$$
the conjugate exponent $q'$ of $q$, and the dual space of $(W^{1,q}_0(\Omega))'=:W^{-1,q'}(\Omega)$, with its operatorial norm denoted by $\|\cdot\|_{-1,q'}$.

\subsection{Two sequences of quasi-eigenvalues}\label{subsubsec:quasieig}
Inspired by \cite{CP} and \cite{LZ}, we define below two sequences, denoted by $(\eta^{(\alpha)}_h)$ and $(\nu^{(\alpha)}_h)$, of quasi-eigenvalues for a $(p,q)$-Laplacian-type operator. Compared with the arguments in \cite{CP,LZ} for the classical $p$-Laplacian, here the arguments are slightly more delicate due to the lack of homogeneity of the operator.
 
Since in this subsection $\alpha\ge 0$ is fixed, for simplicity in notation, throughout this subsection we will drop the superscript $(\alpha)$ and denote the sequences simply by $(\eta_h)$ and $(\nu_h)$.
\bigskip

For $\alpha\in[0,+\infty)$, let us define the $C^1$-functional $\Phi :W^{1,q}_0(\Omega)\to\mathbb R$ as
$$
\Phi(u) =\alpha\|\nabla u\|_p^p+\|\nabla u\|_q^q\quad\mbox{for every }u\in W^{1,q}_0(\Omega).
$$

We are now ready to construct the sequence $(\eta_h)$ and its corresponding sequence of quasi-eigenfunctions $(\varphi_h)$. 

Let us define $\mathcal S:=\{u\in W^{1,q}_0(\Omega)\,:\,\|u\|_q=1\}$ and
$$
\eta_1:=\inf_{u\in \mathcal S}\Phi(u).
$$
Denoted by $\lambda_1$ the first eigenvalue of the Dirichlet $q$-Laplacian, we claim that $\eta_1\ge \lambda_1>0$ is achieved by a function $\varphi_1\in \mathcal S$. 

\begin{proof}[$\bullet$ Proof of the claim.]
Let $(u_n)\subset\mathcal S$ be such that $\Phi(u_n)\to \eta_1$, then for every $n\in\N$
$$
\|\nabla u_n\|_q^q\le \alpha\|\nabla u_n\|_p^p+\|\nabla u_n\|_q^q=\eta_1+o(1),
$$
where $o(1)\to0$ as $n\to\infty$. Then, $(u_n)$ is bounded in the reflexive Banach space $W^{1,q}_0(\Omega)$ and so, up to a subsequence, $u_n\rightharpoonup u$ in $W^{1,q}_0(\Omega)$. 
Now, the function $\Phi$ is convex and continuous w.r.t. the strong topology in $W^{1,q}_0(\Omega)$, then it is weakly lower semicontinuous in $W^{1,q}_0(\Omega)$. Hence,  
\begin{equation}\label{eq:le}
\Phi(u)\le \liminf_{n\to\infty}\Phi(u_n)=\eta_1.
\end{equation} 
On the other hand, by the compact embedding $W^{1,q}_0(\Omega)\hookrightarrow\hookrightarrow L^q(\Omega)$, $u_n\to u$ in $L^q(\Omega)$, and so $u\in\mathcal S$. Therefore, $u$ is an admissible competitor for the infimum defining $\eta_1$, so that the only possibility for \eqref{eq:le} to hold is that $\Phi(u)=\eta_1$. 
\end{proof}

We will denote by $\varphi_1$ a function in $\mathcal{S}$ where $\eta_1$ is attained. Hence, summarizing, $\|\varphi_1\|_q=1$, $\alpha\|\nabla \varphi_1\|_p^p+\|\nabla \varphi_1\|_q^q=\eta_1$, and 
\begin{equation}\label{eq:dis-eta1}
\eta_1 \le \alpha\int_\Omega\left|\nabla \left(\frac{u}{\|u\|_q}\right)\right|^pdx+\int_\Omega\left|\nabla\left(\frac{ u}{\|u\|_q}\right)\right|^qdx\quad	\mbox{for every }u\in W^{1,q}_0(\Omega)\setminus\{0\}.
\end{equation}

\begin{remark}\label{rmk:remark2}
Being $p<q$, \eqref{eq:dis-eta1} implies
$$
\begin{gathered}
\eta_1 \|u\|_q^q\le \alpha\int_\Omega |\nabla u |^pdx+\int_\Omega |\nabla u|^qdx\quad	\mbox{for every }u\in W^{1,q}_0(\Omega)\cap\mathcal B,\\
\eta_1 \|u\|_q^p\le \alpha\int_\Omega |\nabla u |^pdx+\int_\Omega |\nabla u|^qdx\quad	\mbox{for every }u\in W^{1,q}_0(\Omega)\setminus\mathcal B,
\end{gathered}
$$
where $\mathcal B:=\{u\in W^{1,q}_0(\Omega)\,:\,\|u\|_q\le 1\}$. In particular, if $\alpha=0$ or equivalently $p=q$, 
$$
\eta_1 \|u\|_q^q\le (\alpha+1)\int_\Omega |\nabla u|^qdx\quad	\mbox{for every }u\in W^{1,q}_0(\Omega).
$$
We further remark that, if $\alpha>0$ and $q\le p^*$, $\eta_1(=\eta_1^{(\alpha)})>\lambda_1(=\eta_1^{(0)})$. Indeed, by the Sobolev embedding $W^{1,p}_0(\Omega)\hookrightarrow L^q(\Omega)$ ($\|u\|_q\le C_S\|\nabla u\|_p$, for some $C_S>0$, for all $u\in W^{1,p}_0(\Omega)$) and the variational characterization of $\lambda_1$, for every $u\in \mathcal S$
$$
\alpha\|\nabla u\|_p^p+\|\nabla u\|_q^q \ge \alpha C_S^{-p}\|u\|_q^p+ \|\nabla u\|_q^q\ge  \alpha C_S^{-p}+\lambda_1>\lambda_1.
$$
\end{remark}
\smallskip

Related to $\varphi_1$, we can introduce the linear operator $\mathcal L_1:L^{q}(\Omega)\to\mathbb R$ defined as
$$
\mathcal L_1(u):=\int_\Omega|\varphi_1|^{q-2}\varphi_1 u\, dx\quad\mbox{for every }u\in L^{q}(\Omega).
$$
We note that $\mathcal L_1(\varphi_1)=1$, $\mathcal L_1\in L^{q'}(\Omega)$ (in particular, by H\"older's inequality, $\|\mathcal L_1\|_{p'}=1$) and $\mathcal L_1\big|_{W^{1,q}_0(\Omega)}\in W^{-1,q'}(\Omega)$ (in particular, by H\"older's and Poincar\'e's inequalities, $\|\mathcal L_1\|_{-1,q'}\le \lambda_1^{-1/q}$). 

We introduce the new constraint 
$$
\mathcal S_1:=\{u\in\mathcal S\,:\,\mathcal L_1 u=0\}=\mathrm{ker}(\mathcal L_1\big|_S)
$$
and the corresponding constrained infimum
\begin{equation}\label{eq:eta2}
\eta_2:=\inf_{u\in\mathcal S_1}\Phi(u).
\end{equation}
Since $\mathcal S_1\subset \mathcal S$, we have $\eta_1\le\eta_2$. Next, we claim that also the infimum in \eqref{eq:eta2} is attained. 

\begin{proof}[$\bullet$ Proof of the claim.]
The proof is the same as for the previous claim, with the only difference that one has to prove also that $\mathcal L_1(u)=0$, $u$ being the weak limit of the minimizing sequence $(u_n)\subset\mathcal S_1$. This is a consequence of $\mathcal L_1\big|_{W^{1,q}_0(\Omega)}\in W^{-1,q'}(\Omega)$, being $\mathcal L_1(u_n)=0$ for every $n\in\N$. 
\end{proof}

By iterating this procedure, we introduce a sequence of positive numbers $(\eta_h)$, 
a sequence of functions $(\varphi_h)\subset\mathcal S$ and, in correspondence, a sequence of linear operators $(\mathcal L_h)\subset L^{q'}(\Omega)\cap W^{-1,q'}(\Omega)$ defined by
$$
\mathcal L_h u:=\int_\Omega |\varphi_h|^{q-2}\varphi_h u\, dx\quad\mbox{for every }u\in L^{q}(\Omega) \mbox{ and } h\in\mathbb N.
$$
More precisely, denoted $\mathcal S_0:=\mathcal S$, we define the following weakly closed subspaces of $W^{1,q}_0(\Omega)$
$$
\mathcal S_h:=\{u\in\mathcal S\,:\,\mathcal L_1 u=\dots=\mathcal L_h u=0\}=\bigcap_{i=1}^h\mathrm{ker}\big(\mathcal L_i\big|_\mathcal S\big),
$$
and, for every $h\in\N$, the corresponding constrained infimum
\begin{equation}\label{eq:etah}
\eta_h:=\inf_{u\in\mathcal S_{h-1}}\Phi(u),
\end{equation}
each one achieved on the corresponding function $\varphi_h\in\mathcal S_{h-1}$.
From the definition, it  easily follows that  
$$
0<\eta_1\le\dots\le\eta_h\le\eta_{h+1}\le\dots,
$$ 
and 
\begin{equation}\label{eq:properties}
\mathcal L_h\varphi_h=1 \mbox{ for every }h\in\N\quad\mbox{ and }\quad\mathcal L_k\varphi_h=0 \mbox{ for }k=1,\dots,h-1.
\end{equation} 
In particular, $\mathrm{span}\{\varphi_h\}\cap \mathrm{span}\{\varphi_k\}=\{0\}$ if $h\neq k$. Since otherwise, for $k<h$ one could have for a suitable constant $\beta$, $1=\mathcal L_k\varphi_k=\mathcal L_k(\beta\varphi_h)=\beta\mathcal L_k(\varphi_h)=0$, which is absurd.

In the spirit of \cite[Lemma 5.2]{CP}, we prove the following proposition.
\begin{proposition}\label{prop:etahdiv} The sequence $(\eta_h)$ diverges positively.
\end{proposition}
\begin{proof}  Suppose by contradiction that there exists a number $\bar\eta\in(0,\infty)$ such that  
$$
\eta_h\le\bar\eta\quad\mbox{for every }h\in\mathbb N.
$$
As a consequence, for every $h\in\N$
$$
\|\nabla \varphi_h\|_q^q\le \Phi(\varphi_h)=\eta_h\le\bar\eta,
$$
that is $(\varphi_h)$ is bounded in the reflexive Banach space $W^{1,q}_0(\Omega)$. Thus, there exist a subsequence, still denoted by $(\varphi_h)$, and a function $\bar\varphi\in W^{1,q}_0(\Omega)$ such that $\varphi_h\rightharpoonup\bar\varphi$ $\in W^{1,q}_0(\Omega)$ and $\varphi_h\to\bar\varphi$ in $L^{q}(\Omega)$. In particular, $(\varphi_h)$ is a Cauchy sequence in $L^{q}(\Omega)$, thus for any positive $\varepsilon<1$, there is  $h_0\in\N$ such that 
$$
\|\varphi_{h+k}-\varphi_h\|_{q}<\varepsilon\quad\mbox{for every }h\ge h_0,\,k\ge1.
$$
Therefore, \eqref{eq:properties} and the H\"older inequality imply 
$$
\begin{aligned}1&=\mathcal L_h\varphi_h=\mathcal L_h\varphi_h-\mathcal L_h\varphi_{h+k}=|\mathcal L_h(\varphi_{h+k}-\varphi_h)|\le\int_{\Omega}|\varphi_h|^{q-1}|\varphi_{h+k}-\varphi_h| dx\\
&\le \|\varphi_h\|_q\|\varphi_{h+k}-\varphi_h\|_q<\varepsilon,
\end{aligned}
$$
that is a contradiction.
\end{proof}

Let us recall that if $V\subseteq X$ is a closed subspace of a Banach space $X$, a subspace
$W\subset X$ is a (topological) complement of $V$ if $W$ is closed and every $x\in X$ can be uniquely written as $v+w$, with $v\in V$ and $w\in W$; furthermore the projection operators onto $V$ and $W$ are (linear and) continuous. When this happens and $V$ has finite dimension, we say that $W$ has finite codimension, with $\mathrm{codim} W=\mathrm{dim} V$.

\begin{lemma}\label{lem:primadec} 
Let $\alpha\ge 0$ be fixed. For every $h\in\mathbb N$, let us set
\begin{align}
\label{eq:Vh-def}&V_h=V_h^{(\alpha)} :=\mathrm{span}\{\varphi_1,\dots,\varphi_h\},\\
\label{eq:Wh-def}&W_h=W_h^{(\alpha)}:=\bigcap_{i=1}^h\mathrm{ker}\big(\mathcal L_i\big)=\{u\in W^{1,q}_0(\Omega)\,:\,\mathcal L_1 u=\dots=\mathcal L_h u=0\}.
\end{align}
Then 
$$
W^{1,q}_0(\Omega)=V_h\oplus W_h\quad\hbox{ for every } h\in\mathbb N. 
$$
\end{lemma}
\begin{proof} Reasoning as in \cite[Lemma 5.3]{CP}, we point out that for any $u\in V_h$, by the linear independence of $\varphi_1, \dots, \varphi_h$, we can write uniquely 
$$
u=\sum_{i=1}^h c_i\varphi_i,
$$ 
with $(c_1,\dots,c_h)\in\mathbb R^h$ and by \eqref{eq:properties} it results that
$$
\mathcal L_1 u=c_1,\quad\mathcal L_i u=\sum_{j=1}^{i-1}c_j\mathcal L_i\varphi_j+c_i \;\mbox{ for every }i\in\{2,\dots,h\}.
$$
Therefore, given a function $u=\sum_{i=1}^h c_i\varphi_i$ in $V_h$, the following equivalences hold
$$
\begin{aligned}
u\in W_h \quad&\Leftrightarrow\quad \mathcal L_i u=0 \,\mbox{ for every }i\in\{1,\dots, h\}\\
&\Leftrightarrow\quad c_1=\dots=c_h=0\quad\Leftrightarrow\quad u=0.
\end{aligned}
$$
Which means that $V_h\cap W_h=\{0\}$. 
Now, fixed any $u\in W^{1,q}_0(\Omega)$, we put 
$$
c_1:=\mathcal L_1 u,\quad c_i:=\mathcal L_i u-\sum_{j=1}^{i-1}c_j\mathcal L_i\varphi_j \;\mbox{ for every }i\in\{2,\dots,h\}
$$
and 
$$
v:=\sum_{i=1}^h c_i\varphi_i\in V_h.
$$ 
Taken $w:=u-v$, we have that
$$
\mathcal L_1 w=\mathcal L_1(u-v)=\mathcal L_1 u-c_1=0
$$
and for every $i\in\{2,\dots,h\}$
$$
\mathcal L_i w=\mathcal L_i(u-v)=\mathcal L_iu-\sum_{j=1}^{i-1}c_j\mathcal L_i\varphi_j-c_i-\sum_{j=i+1}^h c_j\mathcal L_i\varphi_j=0.
$$
Hence, $w\in W_h$ and we conclude the proof. 
\end{proof}

We remark that, by definition, for each $h\in\mathbb N$, $W_h$ is a closed subspace of $W^{1,q}_0(\Omega)$ of codimension $h$. Moreover, reasoning as in Remark \ref{rmk:remark2}, the following inequalities hold in $W_{h-1}$:
\begin{equation}\label{eq:disug-Wh}
\begin{gathered}
\eta_{h} \|u\|_q^q\le \alpha\int_\Omega |\nabla u |^pdx+\int_\Omega |\nabla u|^qdx\quad	\mbox{for every }u\in W_{h-1}\cap\mathcal B,\\
\eta_{h} \|u\|_q^p\le \alpha\int_\Omega |\nabla u |^pdx+\int_\Omega |\nabla u|^qdx\quad	\mbox{for every }u\in W_{h-1}\setminus\mathcal B,
\end{gathered}
\end{equation}
with $\mathcal B=\{u\in W^{1,q}_0(\Omega)\,:\,\|u\|_q\le 1\}$, and in particular, if $\alpha=0$ or $p=q$, 
\begin{equation}\label{eq:disug-Wh-alpha0}
\eta_{h} \|u\|_q^q\le (\alpha+1)\int_\Omega |\nabla u |^qdx\quad	\mbox{for every }u\in W_{h-1}.
\end{equation}

The sequence of quasi-eigenvalues $(\eta_h)$ satisfies \eqref{eq:disug-Wh}. In order to prove multiplicity results, it is useful to have also a reversed inequality on finite dimensional subspaces of $W^{1,q}_0(\Omega)$. Hence, we introduce $(\nu_h)$, another sequence of quasi-eigenvalues satisfying exactly this property.

For all $h\in\mathbb N$ we set
$$
\mathbb W_h:=\{V\,:\,V\mbox{ subspace of }W^{1,q}_0(\Omega),\,\varphi_1\in V,\,\mathrm{dim}V\ge h\}
$$
and 
\begin{equation}\label{eq:nuh-def}
\nu_h:=\inf_{V\in \mathbb W_h}\sup_{u\in V\setminus\{0\}}\frac{\alpha\|\nabla u\|_p^p+(1-\alpha)\|\nabla u\|_q^q}{\alpha\|u\|_p^p+(1-\alpha)\|u\|_q^q},
\end{equation}
with $\varphi_1$ defined above.
Since $\mathbb W_{h+1}\subset \mathbb W_h$, $\nu_{h}\le\nu_{h+1}$ for every $h$.

\section{Main results}\label{sec:sec3}
We first observe that, by $(f)$ and $(f_\infty)$, for every $\varepsilon>0$  there exists $R_\varepsilon>0$ such that 
\begin{equation}\label{eq:conseq-f-infty}
|f(x,t)|\le \varepsilon|t|^{q-1}+A_\varepsilon\quad\mbox{for a.e. }x\in\Omega \mbox{ and every }t\in\R,
\end{equation}
where $A_\varepsilon:=\left\|\max_{|t|\le R_\varepsilon}|f(\cdot,t)|\right\|_\infty\in (0,\infty)$.

Throughout this section, we will denote by the same symbol $C$ various positive constants whose values are not important for the proof itself and may change from line to line. 

\subsection{Compactness condition}
We introduce here the following operator $A\,:\,W^{1,q}_0(\Omega)\to W^{-1,q'}(\Omega)$ defined as
$$
\langle A u,v\rangle:=\int_\Omega (|\nabla u|^{p-2}+|\nabla u|^{q-2})\nabla u\nabla vdx\quad \mbox{for every }u,\,v\in W^{1,q}_0(\Omega).
$$
 
\begin{lemma}\label{lem:S+}
The operator $-\Delta_p-\Delta_q$ satisfies the $(S_+)$-property, i.e., if $u_n\rightharpoonup u$ in $W^{1,q}_0(\Omega)$ and $\langle Au_n,u_n-u\rangle \to 0$, then $u_n\to u$ in $W^{1,q}_0(\Omega)$. 
\end{lemma}
\begin{proof} This result is contained in \cite[Proposition 2.2]{PS}, but for the sake of clarity, we prefer to report here its proof. Since $\|\nabla u_n\|_p^{p-1}-\|\nabla u\|_p^{p-1}$ and $\|\nabla u_n\|_p-\|\nabla u\|_p$ have the same sign, and the same holds true for the $q$-norms, 
\begin{equation}\label{eq:liminf}
\begin{aligned}
(\|\nabla u_n\|_p^{p-1}-\|\nabla u\|_p^{p-1})&(\|\nabla u_n\|_p-\|\nabla u\|_p)\\
&+(\|\nabla u_n\|_q^{q-1}-\|\nabla u\|_q^{q-1})(\|\nabla u_n\|_q-\|\nabla u\|_q)\ge 0.
\end{aligned}
\end{equation}
On the other hand, for every $u,\,v\in W^{1,q}_0(\Omega)$, by H\"older's inequality 
$$
\langle A u,v\rangle\le \|\nabla u\|_p^{p-1}\|\nabla v\|_p+\|\nabla u\|_q^{q-1}\|\nabla v\|_q.
$$
Therefore, using the previous inequality,
\begin{multline*}
(\|\nabla u_n\|_p^{p-1}-\|\nabla u\|_p^{p-1})(\|\nabla u_n\|_p-\|\nabla u\|_p)
+(\|\nabla u_n\|_q^{q-1}-\|\nabla u\|_q^{q-1})(\|\nabla u_n\|_q-\|\nabla u\|_q)\\
=(\|\nabla u_n\|_p^p+\|\nabla u_n\|_q^q)+(\|\nabla u\|_p^p+\|\nabla u\|_q^q)
-(\|\nabla u_n\|_p^{p-1}\|\nabla u\|_p+\|\nabla u_n\|_q^{q-1}\|\nabla u\|_q)\\
-(\|\nabla u\|_p^{p-1}\|\nabla u_n\|_p+\|\nabla u\|_q^{q-1}\|\nabla u_n\|_q)\\
\le \langle Au_n,u_n\rangle +\langle Au,u\rangle- \langle Au_n,u\rangle-\langle Au,u_n\rangle\\
=\langle Au_n,u_n-u\rangle-\langle Au,u_n-u\rangle=o(1),
\end{multline*}
where $\langle Au_n,u_n-u\rangle=o(1)$ by assumption, and $\langle Au,u_n-u\rangle=o(1)$ because $u_n\rightharpoonup u_n$ both in $W^{1,q}_0(\Omega)$ and in $W^{1,p}_0(\Omega)$ (since $p<q$, $W^{-1,p'}(\Omega)\subset W^{-1,q'}(\Omega)$ and so $u_n \rightharpoonup u$ also in $W^{1,p}_0(\Omega)$). 
Definitely, taking into account also \eqref{eq:liminf}, 
$$
\begin{aligned}(\|\nabla u_n\|_p^{p-1}-\|\nabla u\|_p^{p-1})&(\|\nabla u_n\|_p-\|\nabla u\|_p)\\
&+(\|\nabla u_n\|_q^{q-1}-\|\nabla u\|_q^{q-1})(\|\nabla u_n\|_q-\|\nabla u\|_q)\to 0,
\end{aligned}
$$ 
which implies in particular that $\|\nabla u_n\|_q\to \|\nabla u\|_q$.   
 In conclusion, the convergence of the norms and the weak convergence yield the desired strong convergence $u_n\to u$ in the uniformly convex Banach space $W^{1,q}_0(\Omega)$.
\end{proof}

\begin{lemma}\label{lem:CPS}
Under assumptions $(f)$, $(f_\infty)$, and $(f_\mathrm{cpt})$, the functional $I$ satisfies the (CPS)-condition. 
\end{lemma}
\begin{proof} Let $(u_n)\subset W^{1,q}_0(\Omega)$ be a sequence satisfying 
\begin{equation}\label{eq:CPS-seq}
I(u_n)\to c\in \R\quad\mbox{and}\quad \|I'(u_n)\|_{{-1,q}}(1+\|u_n\|)\to 0 \mbox{ as }n\to\infty.
\end{equation}
We claim that it is enough to show that $(u_n)$ is convergent, provided it is bounded in $W^{1,q}_0(\Omega)$. This is a quite standard consequence of Lemma \ref{lem:S+}, cf. for instance \cite[Lemma 3.1]{FP} and \cite[Lemma 2]{DJM}. However, we prefer to write here all the details of the proof of this claim, because most of the arguments therein will be useful also in the rest of the proof of the present lemma. Suppose that $(\|\nabla u_n\|_q)$ is bounded. Since $W^{1,q}_0(\Omega)$ is a reflexive Banach space, there exist a subsequence, still denoted by $(u_n)$, and a function $u\in W^{1,q}_0(\Omega)$ for which $u_n\rightharpoonup u$ in $W^{1,q}_0(\Omega)$. Moreover, for every $n\in\N$ we have
\begin{equation}\label{eq:I'un-u}
\begin{aligned}
\langle I'(u_n),u_n-u\rangle =& \int_\Omega (|\nabla u_n|^{p-2}+|\nabla u_n|^{q-2})\nabla u_n (\nabla u_n-\nabla u) dx\\
& - \ell_\infty\int_\Omega |u_n|^{q-2}u_n(u_n-u) dx-\int_\Omega f(x,u_n)(u_n-u)dx
\end{aligned}
\end{equation}
and, by \eqref{eq:CPS-seq} and by the boundedness of $(u_n)$ in $W^{1,q}_0(\Omega)$,
$$
|\langle I'(u_n),u_n-u\rangle|\le \|I'(u_n)\|_{-1,q'}\|u_n-u\|\le \|I'(u_n)\|_{-1,q'}\sup_n \|u_n-u\|=o(1),
$$
where $o(1)\to 0$ as $n\to\infty$. 
On the other hand, by the compact embedding $W^{1,q}_0(\Omega) \hookrightarrow\hookrightarrow L^r(\Omega)$,  $u_n\to u$ in $L^r(\Omega)$ for every $r\in [1,q^*)$. Therefore, by H\"older's inequality 
\begin{equation}\label{eq:unto0}
\left|\int_\Omega |u_n|^{q-2}u_n(u_n-u) dx\right|\le \int_\Omega |u_n|^{q-1}|u_n-u| dx\le \|u_n\|_q^{q-1}\|u_n-u\|_q=o(1)
\end{equation}
and, using \eqref{eq:conseq-f-infty} with $\varepsilon=1$ and \eqref{eq:unto0},
\begin{equation}\label{eq:fto0}
\begin{aligned}
\left|\int_\Omega f(x,u_n)(u_n-u)dx\right|&\le \int_\Omega \left(|u_n|^{q-1}|u_n-u|+A_1 |u_n-u| \right)dx\\
&= o(1)+A_1\|u_n-u\|_1=o(1).
\end{aligned}
\end{equation}
Hence, inserting the last three estimates in \eqref{eq:I'un-u}, we get 
\begin{equation}\label{eq:to0}
\int_\Omega (|\nabla u_n|^{p-2}+|\nabla u_n|^{q-2})\nabla u_n (\nabla u_n-\nabla u) dx=o(1).
\end{equation}
Therefore, by Lemma \ref{lem:S+}, $u_n\to u$ in $W^{1,q}_0(\Omega)$. 

It remains to prove that $(u_n)$ is bounded in $W^{1,q}_0(\Omega)$. We will follow the guidelines of \cite[Proposition 3.1-(i)]{BCS} and of \cite[Proposition 3.1-(ii)]{LZ}.\smallskip 

$\bullet$ We first consider the non-resonant case in which $(f_\mathrm{nr})$ holds.
We argue by contradiction and suppose that there exists a subsequence, still denoted by $(u_n)$, such that  $\|u_n\|\to\infty$. Thus, without loss of generality, we can assume that $\|u_n\|>0$ for every $n$ and define $w_n:=u_n/\|u_n\|$. Clearly, $(w_n)$ is bounded in $W^{1,q}_0(\Omega)$, hence up to a subsequence, $w_n\rightharpoonup w$ in $W^{1,q}_0(\Omega)$ and $w_n\to w$ in $L^r(\Omega)$ for every $r\in [1,q^*)$ and some $w\in W^{1,q}_0(\Omega)$. We claim that $w\neq 0$.

Indeed, as a consequence of the second convergence in \eqref{eq:CPS-seq}, the following relation holds
\begin{equation}\label{eq:conseq1-CPS}
\int_\Omega (|\nabla u_n|^p+|\nabla u_n|^q)dx=\ell_\infty\int_\Omega |u_n|^q dx +\int_\Omega f(x,u_n)u_ndx+o(1).
\end{equation}
Dividing \eqref{eq:conseq1-CPS} by $\|u_n\|^q=\|\nabla u_n\|_q^q$, we get 
\begin{equation}\label{eq:conseq1'-CPS}
\frac{\|\nabla u_n\|_p^p}{\|\nabla u_n\|_q^q}+1=\ell_\infty\|w_n\|^q _q +\int_\Omega \frac{f(x,u_n)u_n}{\|\nabla u_n\|_q^q}dx+o(1).
\end{equation}
Moreover, by H\"older's inequality, 
\begin{equation}\label{eq:goes-to0}
\frac{\|\nabla u_n\|_p^p}{\|\nabla u_n\|_q^q}\le \frac{C}{\|\nabla u_n\|_q^{q-p}}=\frac{C}{\|u_n\|^{q-p}}=o(1).
\end{equation}
Assume by contradiction that $w=0$. Then, since $w_n\to w$ in $L^q(\Omega)$, by \eqref{eq:conseq1'-CPS} we get
\begin{equation}\label{eq:conseq-w=0}
\int_\Omega \frac{f(x,u_n)u_n}{\|\nabla u_n\|_q^q}dx=1+o(1).
\end{equation}
On the other hand, by \eqref{eq:conseq-f-infty}, \eqref{eq:embed-Wq-in-Wr} with $r=1$, and Poincar\'e's inequality, we have
$$
\left|\int_\Omega\frac{f(x,u_n)u_n}{\|\nabla u_n\|_q^q}dx\right|\le \|w_n\|_q^q+A_1\frac{\|u_n\|_1}{\|\nabla u_n\|_q^q}\le \|w_n\|_q^q+\frac{C}{\|\nabla u_n\|_q^{q-1}}=o(1).
$$
This contradicts \eqref{eq:conseq-w=0} and proves that $w\neq 0$. 

Another consequence of \eqref{eq:CPS-seq} is the following
\begin{equation}\label{eq:conseq2-CPS}
\begin{aligned}
\int_\Omega &(|\nabla u_n|^{p-2}+|\nabla u_n|^{q-2})\nabla u_n\nabla \varphi  \,dx \\
&=  \ell_\infty\int_\Omega |u_n|^{q-2}u_n \varphi \,dx +\int_\Omega f(x,u_n)\varphi \,dx+o(1)\mbox{ for every } \varphi\in W^{1,q}_0(\Omega).
\end{aligned}
\end{equation}
We divide \eqref{eq:conseq2-CPS} by $\|u_n\|^{q-1}=\|\nabla u_n\|_q^{q-1}$ to get for every $\varphi\in W^{1,q}_0(\Omega)$
\begin{equation}\label{eq:conseq2'-CPS}
\begin{aligned}
\int_\Omega\left(\frac{|\nabla u_n|^{p-2}\nabla u_n}{\|\nabla u_n\|_q^{q-1}}+|\nabla w_n|^{q-2}\nabla w_n\right)&\nabla \varphi \,dx\\
& \hspace{-3cm} = \ell_\infty\int_\Omega |w_n|^{q-2}w_n \varphi \,dx+\int_\Omega\frac{f(x,u_n)\varphi}{\|\nabla u_n\|_q^{q-1}}dx+o(1).
\end{aligned}
\end{equation}
Now, by H\"older's inequality and \eqref{eq:embed-Wq-in-Wr}, 
\begin{equation}\label{eq:nablato0}
\left|\int_\Omega\frac{|\nabla u_n|^{p-2}\nabla u_n\nabla \varphi }{\|\nabla u_n\|_q^{q-1}}dx \right|\le 
\int_\Omega\frac{|\nabla u_n|^{p-1}|\nabla \varphi|}{\|\nabla u_n\|_q^{q-1}}dx\le C\frac{\|\nabla \varphi\|_p}{\|\nabla u_n\|_q^{q-p}}=o(1).
\end{equation}
Thus, \eqref{eq:conseq2'-CPS} evaluated at $\varphi=w$ can be written as 
\begin{equation}\label{eq:conseq2-CPS-w}
\int_\Omega |\nabla w_n|^{q-2}\nabla w_n \nabla w \,dx  = \ell_\infty\int_\Omega |w_n|^{q-2}w_n w \,dx+\int_\Omega\frac{f(x,u_n)w}{\|\nabla u_n\|_q^{q-1}}dx+o(1).
\end{equation}
On the other hand, taking into account \eqref{eq:goes-to0} and being $\int_\Omega|\nabla w_n|^qdx=1$ for every $n$, \eqref{eq:conseq1'-CPS} reads as
\begin{equation}\label{eq:conseq1''-CPS}
\int_\Omega |\nabla w_n|^{q-2}\nabla w_n \nabla w_n dx  = \ell_\infty\int_\Omega |w_n|^{q-2}w_nw_n dx+\int_\Omega\frac{f(x,u_n)w_n}{\|\nabla u_n\|_q^{q-1}}dx+o(1).
\end{equation}
Thus, subtracting \eqref{eq:conseq1''-CPS} and \eqref{eq:conseq2-CPS-w}, we get 
\begin{equation}\label{eq:subtr}
\begin{aligned}
\int_\Omega &|\nabla w_n|^{q-2}\nabla w_n (\nabla w_n-\nabla w) dx \\
& = \ell_\infty\int_\Omega |w_n|^{q-2}w_n(w_n-w) dx+\int_\Omega\frac{f(x,u_n)}{\|\nabla u_n\|_q^{q-1}}(w_n-w)dx+o(1).
\end{aligned}
\end{equation}
Arguing as in \eqref{eq:unto0} and \eqref{eq:fto0}, we get  
$$
\begin{gathered}
\left|\int_\Omega |w_n|^{q-2}w_n(w_n-w) dx\right|\le \|w_n\|_q^{q-1}\|w_n-w\|_q=o(1),\\
\left|\int_\Omega\frac{f(x,u_n)}{\|\nabla u_n\|_q^{q-1}}(w_n-w)dx\right| \le \int_\Omega \left(|w_n|^{q-1}|w_n-w|+A_1\frac{|w_n-w|}{\|\nabla u_n\|_q^{q-1}}\right)dx =o(1),
\end{gathered}
$$
Definitely, inserting the last two relations in \eqref{eq:subtr}, we have that also 
\begin{equation}\label{eq:forS+}
\int_\Omega |\nabla w_n|^{q-2}\nabla w_n (\nabla w_n-\nabla w) dx=o(1). 
\end{equation}
Since $-\Delta_q$ satisfies the $(S_+)$-property (cf. for instance \cite[Lemma 2.5]{CPV}) and $w_n \rightharpoonup w$ in $W^{1,q}_0(\Omega)$, \eqref{eq:forS+} implies that $w_n\to w$ in $W^{1,q}_0(\Omega)$. 
Finally, by \eqref{eq:conseq2'-CPS} and \eqref{eq:nablato0}, for every $\varphi\in W^{1,q}_0(\Omega)$
\begin{equation}\label{eq:eigen}
\int_\Omega |\nabla w_n|^{q-2}\nabla w_n \nabla \varphi \,dx = \ell_\infty\int_\Omega |w_n|^{q-2}w_n \varphi \,dx+\int_\Omega\frac{f(x,u_n)\varphi}{\|\nabla u_n\|_q^{q-1}}dx+o(1).
\end{equation}
By \eqref{eq:conseq-f-infty}, the fact that $(u_n)$ is convergent and so bounded in $L^q(\Omega)$, and $\|\nabla u_n\|_q\to \infty$,
$$
\int_\Omega\frac{f(x,u_n)\varphi}{\|\nabla u_n\|_q^{q-1}}dx\le \frac{\|u_n\|_q^{q-1}\|\varphi\|_q}{\|\nabla u_n\|_q^{q-1}}+o(1)\le\frac{C\|\varphi\|_q}{\|\nabla u_n\|_q^{q-1}}+o(1)= o(1).
$$
And so, passing to the limit in \eqref{eq:eigen} we get, by the Dominated Convergence Theorem,
$$
\int_\Omega |\nabla w|^{q-2}\nabla w  \nabla \varphi \,dx = \ell_\infty\int_\Omega |w |^{q-2}w  \varphi \,dx\quad\mbox{for every }\varphi\in W^{1,q}_0(\Omega),
$$
meaning that $\ell_\infty$ is an eigenvalue of $-\Delta_q$. This contradicts $(f_\mathrm{nr})$ and concludes the proof of this case.\smallskip

$\bullet$ We now consider the resonant case and assume the validity of $(f_\mathrm{r})$. We observe that in this case $\ell_\infty>0$. In this part of the proof, for brevity, we will write the nonlinear terms $\ell_\infty |t|^{q-2}t+f(x,t)$ as $g(x,t)$.
Let $G(x,t):=\int_0^t g(x,s) ds$. It is straightforward to check that $g(x,t)t-qG(x,t)=f(x,t)t-qF(x,t)$ for  a.e. $x\in\Omega$ and every $t\in\R$. Thus, the limit in $(f_\mathrm{r})$ can be rewritten equivalently in terms of $g$ as follows
\begin{equation}\label{eq:fresonant-g}
\lim_{|t|\to\infty}(g(x,t)t-qG(x,t))=+\infty\quad\mbox{uniformly in }x\in\Omega.
\end{equation}
As a consequence, there exists $T_0>0$ such that 
\begin{equation}\label{eq:conseq-fres-g}
g(x,t)t-qG(x,t)\ge 0\quad\mbox{for a.e. }x\in\Omega\mbox{ and for every }|t|\ge T_0.
\end{equation}
By virtue of $(f)$, there exists $C_0>0$ such that 
\begin{equation}\label{eq:gG}
\int_{\{|u_n|\le T_0\}}(g(x,u_n)u_n-qG(x,u_n))dx\ge -C_0,
\end{equation}
where for simplicity in notation we have denoted $\{|u_n|\le T_0\}= \{x\in\Omega\,:\,|u_n(x)|\le T_0\}$. 
Now, by \eqref{eq:CPS-seq}, 
$$
\begin{aligned}
qc+o(1)& =qI(u_n)-\langle I'(u_n),u_n\rangle\\
&=\left(\frac{q}{p}-1\right)\int_\Omega|\nabla u_n|^pdx+\int_\Omega (g(x,u_n)u_n-qG(x,u_n))dx
\end{aligned}
$$
Hence, by \eqref{eq:gG} and being $q>p$,
\begin{equation}\label{eq:qc}
qc+o(1)\ge \int_\Omega (g(x,u_n)u_n-qG(x,u_n))dx\ge \int_{\{|u_n|\ge T_0\}} (g(x,u_n)u_n-qG(x,u_n))dx-C_0.
\end{equation}
Now, let $K>0$ be a constant to be specified later. By \eqref{eq:fresonant-g}, there exists $T_K\ge T_0>0$ such that 
\begin{equation}\label{eq:conseq-fres-g-K}
g(x,t)t-qG(x,t)\ge K\quad\mbox{for a.e. }x\in\Omega\mbox{ and for every }|t|\ge T_K.
\end{equation}

Thus, continuing from \eqref{eq:qc}, we have
\begin{equation}\label{eq:qc1}
qc+o(1)\ge K|{\{|u_n|\ge T_K\}}|-C_0.
\end{equation}
On the other hand, for every $r>q>p$, by \eqref{eq:CPS-seq},
$$
\begin{aligned}
I(u_n)-\frac{1}{r}\langle I'(u_n),u_n\rangle &=
\left(\frac{1}{p}-\frac{1}{r}\right)\|\nabla u_n\|_p^p\\
&\phantom{=}+\left(\frac{1}{q}-\frac{1}{r}\right)\|\nabla u_n\|_q^q-\int_\Omega[G(x,u_n)-\frac{1}{r}g(x,u_n)u_n]dx\\
&=c+o(1)
\end{aligned}
$$
Since, by $(f)$, $\left|\int_{\{|u_n|\le T_K\}}[G(x,u_n)-\frac{1}{r}g(x,u_n)u_n]dx\right|\le C_K$ for some $C_K>0$, 
$$
\begin{aligned}
c+o(1)&\ge \left(\frac{1}{q}-\frac{1}{r}\right)\|\nabla u_n\|_q^q-\int_{\{|u_n|\ge T_K\}}[G(x,u_n)-\frac{1}{r}g(x,u_n)u_n]dx-C_K\\
&\ge \left(\frac{1}{q}-\frac{1}{r}\right)\|\nabla u_n\|_q^q-\int_{\{|u_n|\ge T_K\}}\left[\frac{g(x,u_n)u_n-K}{q}-\frac{1}{r}g(x,u_n)u_n\right]dx-C_K\\
&\ge \left(\frac{1}{q}-\frac{1}{r}\right)\|\nabla u_n\|_q^q-\left(\frac{1}{q}-\frac{1}{r}\right)\int_{\{|u_n|\ge T_K\}}g(x,u_n)u_n dx-C_K\\
&\ge \left(\frac{1}{q}-\frac{1}{r}\right)\left[\|\nabla u_n\|_q^q-\int_{\{|u_n|\ge T_K\}}[(\ell_\infty+1)|u_n|^q+A_1|u_n|]dx\right]-C_K\\
&\ge \left(\frac{1}{q}-\frac{1}{r}\right)\left[\|\nabla u_n\|_q^q-\int_{\{|u_n|\ge T_K\}}(\ell_\infty+1)|u_n|^q dx-A_1|\Omega|^\frac{1}{q'}\|u_n\|_q\right]-C_K\\
&\ge \left(\frac{1}{q}-\frac{1}{r}\right)\left[\|\nabla u_n\|_q^q-\int_{\{|u_n|\ge T_K\}}(\ell_\infty+1)|u_n|^q dx-C\|\nabla u_n\|_q\right]-C_K
\end{aligned}
$$
where we have used \eqref{eq:conseq-fres-g-K}, \eqref{eq:conseq-f-infty}, H\"older's and Poincar\'e's inequalities. Now, fix  $s\in (q,q^*)$. By H\"older's inequality, 
$$
\int_{\{|u_n|\ge T_K\}}|u_n|^q dx\le \left|\{|u_n|\ge T_K\}\right|^{\frac{s-q}s}\|u_n\|_s^q.
$$
In view of the embedding $W^{1,q}_0(\Omega)\hookrightarrow L^s(\Omega)$, let $C_S>0$ be the best constant such that the following inequality holds for every $u\in W^{1,q}_0(\Omega)$
$$
\|u\|_s^q\le C_S \|\nabla u\|_q^q.
$$
Thus, continuing the previous estimate, we get 
\begin{equation}\label{eq:quasi.final}
\begin{aligned}
c+o(1)&\ge \left(\frac{1}{q}-\frac{1}{r}\right)\left(1-(\ell_\infty+1)C_S|\{|u_n|\ge T_K\}|^{\frac{s-q}s}\right)\|\nabla u_n\|_q^q -C\|\nabla u_n\|_q-C_K.
\end{aligned}
\end{equation}
Now, it suffices to choose $K$ in such a way that 
$$
1-(\ell_\infty+1)C_S|\{|u_n|\ge T_K\}|^{\frac{s-q}s}\ge\frac12+o(1).
$$ 
This is possible thanks to \eqref{eq:qc1}, taking $K=(qc+C_0)[2C_S(\ell_\infty+1)]^{\frac{s}{s-q}}$.
In conclusion, \eqref{eq:quasi.final} gives 
$$
C\|\nabla u_n\|_q+c+C_K+o(1)\ge C'\|\nabla u_n\|_q^q
$$ 
for a suitable positive constant $C'$. This proves the boundedness of $(\|u_n\|)$ and concludes the proof.
\end{proof}

\subsection{Geometry conditions and proof of Theorem \ref{thm:main}}

\subsubsection{The case $(f_0^-)$.}\label{subsubsec:ell0<0}
In this part, we consider the sequences of quasi-eigenvalues $(\eta^{(\alpha)}_h)$ and $(\nu^{(\alpha)}_h)$ introduced in Section \ref{sec:sec2} related to the classical $q$-Laplacian case, namely with $\alpha=0$. Since $\alpha=0$ is fixed, we will drop the superscript $(0)$ throughout the present Subsection \ref{subsubsec:ell0<0}. 

\begin{lemma}\label{lem:palletta>0}
Assume that  $(f)$, $(f_\infty)$, and $(f_0^-)$ hold. If  
\begin{equation}\label{eq:cond-l0<0}
\ell_\infty+\ell_0<\eta_h\quad\mbox{for some }h\in\N,
\end{equation}
there exist $\rho>0$ sufficiently small and a constant $c_0>0$ such that
$$
I(u)\ge c_0\quad\mbox{for every }u\in S_\rho\cap W,
$$
where $S_\rho=\{u\in W^{1,q}_0(\Omega)\,:\,\|\nabla u\|_q=\rho\}$ and $W\subset W^{1,q}_0(\Omega)$ is a closed subspace such that $\mathrm{codim}W=h-1$.
\end{lemma}
\begin{proof}
$\bullet$  We first consider the case $\ell_0>-\infty$. By $(f)$, $(f_\infty)$, and $(f_0^-)$, for every $\varepsilon>0$ and every $s\in[0,q^*-q)$ there exists $a_\varepsilon>0$ such that  
\begin{equation}\label{eq:conseq-hp-f}
F(x,t)\le \frac{\varepsilon+\ell_0}{q}|t|^q+a_\varepsilon |t|^{s+q}\quad \mbox{for a.e. }x\in \Omega \mbox{ and every }t\in\R.
\end{equation}
Therefore, for every $u\in W^{1,q}_0(\Omega)$, by the Sobolev embedding $W^{1,q}_0(\Omega)\hookrightarrow L^{s+q}(\Omega)$
$$
\begin{aligned}
I(u)&\ge \frac{1}{p}\|\nabla u\|_p^p+\frac{1}{q}\|\nabla u\|_q^q-\frac{\ell_\infty+\varepsilon+\ell_0}{q}\|u\|_q^q- a_\varepsilon\|u\|_{s+q}^{s+q}\\
&\ge \frac{1}{q}\|\nabla u\|_q^q-\frac{\ell_\infty+\varepsilon+\ell_0}{q}\|u\|_q^q- C\|\nabla u\|_{q}^{s+q}.
\end{aligned}
$$
Moreover, we consider the closed subspace $W_{h-1}$ defined in \eqref{eq:Wh-def}. By \eqref{eq:disug-Wh-alpha0}, the following inequality holds
$$
\eta_{h}\|u\|_q^q\le \|\nabla u\|_q^q \quad\mbox{for every }u\in W_{h-1}.
$$
Hence, fixed $\varepsilon\in(\max\{0,-\ell_\infty-\ell_0\}, \eta_h-\ell_\infty-\ell_0)$, by \eqref{eq:cond-l0<0}, there exists $\rho>0$ so small that for every $u\in S_\rho\cap W_{h-1}$, $$
I(u)\ge \frac1q\left(1-\frac{\ell_\infty+\varepsilon+\ell_0}{\eta_h}\right)\|\nabla u\|_q^q- C\|\nabla u\|_{q}^{s+q}\ge c_0
$$
for some $c_0>0$ and the conclusion follows with $W=W_{h-1}$ being, by Lemma \ref{lem:primadec}, $\mathrm{codim}W_{h-1}=h-1$.\smallskip

$\bullet$ If $\ell_0=-\infty$, clearly condition \eqref{eq:cond-l0<0} is trivially verified also when $h=1$. Again by the assumptions on $f$, for every $M>0$ there exists $a_M>0$ such that 
$$
F(x,t)\le -\frac{M}{q}|t|^q+a_M |t|^{s+q}\quad \mbox{for a.e. }x\in \Omega \mbox{ and every }t\in\R.
$$
Hence, for every $u\in W^{1,q}_0(\Omega)$, taking $M>\max\{0,\ell_\infty\}$, 
$$
\begin{aligned}
I(u)&\ge  \frac{1}{q}\|\nabla u\|_q^q-\frac{\ell_\infty-M}{q}\|u\|_q^q- C\|\nabla u\|_{q}^{s+q}\\
&\ge \frac{1}{q}\|\nabla u\|_q^q- C\|\nabla u\|_{q}^{s+q},
\end{aligned}
$$
and the conclusion follows as in the previous case in $S_\rho\cap W_0=S_\rho\cap W^{1,q}_0(\Omega)=S_\rho$ with $\rho>0$ small enough.
\end{proof}

\begin{remark}\label{rmk:remark-eta1}
A careful inspection in the proof of the previous lemma shows that the result holds true also under the assumption 
\begin{equation}\label{eq:alternative}
-\infty<\ell_\infty+\ell_0<\eta_h^{(1)}.
\end{equation}
Indeed, in this case, by \eqref{eq:disug-Wh}, for every $u\in\{ u\in W^{(1)}_{h-1}\,:\,\|u\|_q\le 1\}$
$$
\begin{aligned}
I(u)&\ge \frac{1}{q}(\|\nabla u\|_p^p+\|\nabla u\|_q^q)-\frac{\ell_\infty+\varepsilon+\ell_0}{q}\|u\|_q^q- C\|\nabla u\|_{q}^{s+q}\\
&\ge \frac1q\left(1-\frac{\ell_\infty+\varepsilon+\ell_0}{\eta^{(1)}_h}\right)(\|\nabla u\|_p^p+\|\nabla u\|_q^q)- C\|\nabla u\|_{q}^{s+q}\\
&\ge \frac1q\left(1-\frac{\ell_\infty+\varepsilon+\ell_0}{\eta^{(1)}_h}\right)\|\nabla u\|_q^q- C\|\nabla u\|_{q}^{s+q}.
\end{aligned}
$$
Now, if  $\|\nabla u\|_q=\rho$ for $0<\rho<\lambda_1^{1/q}$, by Poincar\'e's inequality $\|u\|_q\le 1$, and so $S_\rho\cap W^{(1)}_{h-1}\cap\{\|u\|_q\le 1\}=S_\rho\cap W^{(1)}_{h-1}$. Hence, for $\rho$ small enough, $I(u)\ge c_0>0$ in $S_\rho\cap W^{(1)}_{h-1}$.
We observe that, by Remark \ref{rmk:remark2}, at least for $h=1$ and $q\le p^*$, $\eta_1^{(1)}>\eta_1^{(0)}$, and so condition \eqref{eq:alternative} is weaker than  $\ell_\infty+\ell_0<\eta_h^{(0)}$.
\end{remark}

\begin{lemma}\label{lem:to-infty}
Assume that  $(f)$ and $(f_\infty)$ hold. If, for some $k\in\N$, $\ell_\infty>\nu_k$, there exist a $k$-dimensional closed subspace $V\subset W^{1,q}_0(\Omega)$, and a constant $c_\infty\in (c_0,\infty)$ such that 
$$
I(u)\le c_\infty\quad\mbox{for every }u\in V.
$$
\end{lemma}
\begin{proof}
By \eqref{eq:conseq-f-infty}, for every $\varepsilon>0$ and $u\in W^{1,q}_0(\Omega)$, we have 
$$
\begin{aligned}
I(u)&\le \frac{1}{p}\|\nabla u\|_p^p+\frac{1}{q}\|\nabla u\|_q^q-\frac{\ell_\infty-\varepsilon/2}{q}
\|u\|_q^q+A_{\varepsilon/2}\|u\|_1\\
&\le C\|\nabla u\|_q^p+\frac{1}{q}\|\nabla u\|_q^q-\frac{\ell_\infty-\varepsilon/2}{q}
\|u\|_q^q+C'\|u\|_q
\end{aligned}
$$
where in the last estimate we have used H\"older's inequality, being $1<p<q$. Now, by \eqref{eq:nuh-def}, there exists a $k$-dimensional subspace $V_k^\varepsilon\subset W^{1,q}_0(\Omega)$ such that for every $u\in V_k^\varepsilon$
$$
\|\nabla u\|_q<\left(\nu_k+\frac{\varepsilon}{2}\right)^{1/q}\|u\|_q.
$$
Therefore, for every $u\in V_k^\varepsilon$
$$
I(u)\le \frac{\nu_k-\ell_\infty+\varepsilon}{q}\|u\|_q^q+C\left(\nu_k +\frac{\varepsilon}{2}\right)^{p/q}\|u\|_q^p+C'\|u\|_q
$$
Hence, being by assumption $\ell_\infty>\nu_k$, for $\varepsilon$ sufficiently small $I|_{V_k^\varepsilon}(u)\to-\infty$ as $\|u\|_q\to\infty$. Being $V_k^\varepsilon$ a finite dimensional linear space, of dimension $k$, the proof is concluded by taking $V=V_k^\varepsilon$ and enlarging if necessary $c_\infty$ to make it greater than $c_0$.
\end{proof}

We are now ready to prove the first part of Theorem \ref{thm:main}.

\begin{proof}[$\bullet$ Proof of Theorem \ref{thm:main}-$(H_-)$] In view of Lemmas \ref{lem:CPS}, \ref{lem:palletta>0} and \ref{lem:to-infty}, we can apply Theorem \ref{thm:multiplicity} with $X=W^{1,q}_0(\Omega)$, $W=W_{h-1}$, and $V=V_k^\varepsilon$.
\end{proof}

\subsubsection{The case $(f_0^+)$.}
In this part, we consider the sequences of quasi-eigenvalues $(\eta^{(0)}_h)$ and $(\nu^{(1)}_h)$ introduced in Section \ref{sec:sec2}.

\begin{lemma}\label{lem:palletta>0>0}
Assume that  $(f)$, $(f_\infty)$, and $(f_0^+)$ hold. If  
\begin{equation}\label{eq:cond-l0>0}
\ell'_0>\nu^{(1)}_k\quad\mbox{for some }k\in\N,
\end{equation}
there exist a $k$-dimensional closed subspace $V\subset W^{1,q}_0(\Omega)$, and two positive constants $\rho>0$ and $c_0>0$ such that
$$
-I(u)\ge c_0\quad\mbox{for every }u\in S_\rho\cap V.
$$
In particular, if $\ell'_0=\infty$, the conclusion holds for every $k\in\N$.
\end{lemma}
\begin{proof} $\bullet$ We first consider the case $\ell'_0<\infty$. By $(f)$, $(f_\infty)$, and $(f_0^+)$, for every $\varepsilon>0$ there exists $a'_\varepsilon>0$ such that  
\begin{equation}\label{eq:conseq-hp-f>}
F(x,t)\ge \frac{\ell'_0-\varepsilon}{p}|t|^p-a'_\varepsilon |t|^{q}\quad \mbox{for a.e. }x\in \Omega \mbox{ and every }t\in\R.
\end{equation}
Therefore, fix $0<\varepsilon<(\ell'_0-\nu^{(1)}_k)/2$ to get for every $u\in W^{1,q}_0(\Omega)$
\begin{align*}
-I(u) \ge -\frac{1}{p}\|\nabla u\|_p^p-\frac{1}{q}\|\nabla u\|_q^q + \frac{\ell_\infty}{q}\|u\|_q^q + \frac{\ell'_0-\varepsilon}{p}\|u\|_p^p - a'_{\varepsilon}\|u\|_{q}^{q}.
\end{align*}
By \eqref{eq:nuh-def}, there exists a $k$-dimensional closed subspace $V_k^{\varepsilon,(1)}\subset W^{1,q}_0(\Omega)$ such that for every $u\in V_k^{\varepsilon,(1)}$
\begin{equation}\label{eq:disug-inV'_k}
(\nu^{(1)}_k+\varepsilon)\|u\|_p^p >\|\nabla u\|_p^p.
\end{equation}
Now, let $\rho>0$. For every $u\in V_k^{\varepsilon,(1)}\cap S_\rho$ it holds
\begin{align*}
-I(u)&\ge \frac{1}{p}\left(\frac{\ell'_0-\varepsilon}{\nu^{(1)}_k+\varepsilon}-1\right)\|\nabla u\|_p^p-\frac{1}{q}\|\nabla u\|_q^q - \left(a'_{\varepsilon}-\frac{\ell_\infty}{q}\right)\|u\|_{q}^{q}\\
&\ge \frac{1}{p}\left(\frac{\ell'_0-\varepsilon}{\nu^{(1)}_k+\varepsilon}-1\right)C\rho^p-\left[\frac{1}{q} + \left(a'_{\varepsilon}-\frac{\ell_\infty}{q}\right)C'\right]\rho^{q},
\end{align*}
where $C,\,C'>0$ are two constants arising in the equivalence of the norms in the finite dimensional space $V_k^{\varepsilon,(1)}$.
Since $p<q$ and in view of the choice of $\varepsilon$, there exists $\rho>0$ so small that for every $u\in V_k^{\varepsilon,(1)}\cap S_\rho$, 
$$
-I(u)\ge c_0\quad\mbox{for some }c_0>0
$$
and the conclusion follows taking $V=V_k^{\varepsilon,(1)}$.\smallskip

$\bullet$ If $\ell'_0=\infty$, clearly condition $\ell'_0>\nu^{(1)}_k$ is satisfied for every $k\in\N$. Again, by the assumptions on $f$, for every $M>0$, there exists $a'_M>0$ such that 
$$
F(x,t)\ge \frac{M}{p}|t|^p - a'_M |t|^{q}\quad \mbox{for a.e. }x\in \Omega \mbox{ and every }t\in\R.
$$
Thus for every $k$, choosing $M>\nu^{(1)}_k$, we get for every $u\in W^{1,q}_0(\Omega)$
$$
-I(u)\ge -\frac{1}{p}\|\nabla u\|_p^p -\frac{1}{q}\|\nabla u\|_q^q +\frac{\ell_\infty}{q}\|u\|_q^q +\frac{M}{p}\|u\|_p^p - a'_M\|u\|_{q}^{q}.
$$
Now, fix $0 < \varepsilon < M -\nu^{(1)}_k$. By \eqref{eq:nuh-def}, there exists a $k$-dimensional closed subspace $V_k^{\varepsilon,(1)}\subset W^{1,q}_0(\Omega)$ such that \eqref{eq:disug-inV'_k} holds.
Therefore, for $\rho>0$, 
\[
-I(u)\ge \frac{1}{p}\left(\frac{M}{\nu^{(1)}_k+\varepsilon}-1\right)C\rho^p-\left[\frac{1}{q}+ \left(a'_M-\frac{\ell_\infty}{q}\right)C'\right]\rho^{q}\;\mbox{ for every $u\in V_k^{\varepsilon,(1)}\cap S_\rho$,}
\]
where again $C,\,C'>0$ are two constants arising in the equivalence of the norms in the finite dimensional space $V_k^{\varepsilon,(1)}$.
Then, taking $\rho>0$ sufficiently small, the conclusion follows as in the previous case.
\end{proof}

\begin{lemma}\label{lem:to-infty>0}
Assume that  $(f)$ and $(f_\infty)$ hold. If, for some $h\in\N$, $\ell_\infty<\eta^{(0)}_h$, there exist a closed subspace $W\subset W^{1,q}_0(\Omega)$ of codimension $h-1$, and a constant $c_\infty\in (c_0,\infty)$ such that 
$$
-I(u)\le c_\infty\quad\mbox{for every }u\in W.
$$
\end{lemma}
\begin{proof}
By \eqref{eq:conseq-f-infty}, for every $\varepsilon>0$ and $u\in W^{1,q}_0(\Omega)$, we have 
\begin{align*}
-I(u)&\le -\frac{1}{p}\|\nabla u\|_p^p-\frac{1}{q}\|\nabla u\|_q^q+\frac{\ell_\infty+\varepsilon}{q}
\|u\|_q^q+A_\varepsilon\|u\|_1\\
&\le -\frac{1}{q}\|\nabla u\|_q^q+\frac{\ell_\infty+\varepsilon}{q}
\|u\|_q^q+C\|u\|_q.
\end{align*}
Now, by \eqref{eq:disug-Wh-alpha0}, for every $u\in W^{(0)}_{h-1}$
$$
-I(u)\le -\frac{1}{q}\left(\eta^{(0)}_h-\ell_\infty-\varepsilon\right)\|u\|_q^q+C\|u\|_q.
$$
Hence, being by assumption $\ell_\infty<\eta^{(0)}_h$, for $0<\varepsilon<\eta_h^{(0)}-\ell_\infty$ we have that $-I|_{W^{(0)}_{h-1}}(u)\to-\infty$ as $\|u\|_q\to\infty$. Being $W^{(0)}_{h-1}$ a closed subspace of codimension $h-1$, the proof is concluded by taking $W:=W^{(0)}_{h-1}$ and enlarging if necessary $c_\infty$ to make it greater than $c_0$.
\end{proof}

We are now ready to prove the second part of Theorem \ref{thm:main}.

\begin{proof}[$\bullet$ Proof of Theorem \ref{thm:main}-$(H_+)$] In view of Lemmas \ref{lem:CPS}, \ref{lem:palletta>0>0} and \ref{lem:to-infty>0}, we can apply Theorem \ref{thm:multiplicity} to the functional $-I$, with $X=W^{1,q}_0(\Omega)$, $W=W^{(0)}_{h-1}$, and $V=V_k^{\varepsilon,(1)}$.
\end{proof}
 
\begin{remark}\label{rmk:lastremark}
It has been proved in \cite{BCS} that the sequences $(\eta^{(0)}_h)$ and $(\nu^{(0)}_h)$ are increasing and divergent and that,  for every $h\in\N$, $\eta^{(0)}_h\le\nu^{(0)}_h$. Therefore, the two conditions required in $(H_-)$, when $\ell_0<0$, can be written as the following chain of inequalities 
$$
\ell_\infty+\ell_0<\eta^{(0)}_h \le\eta^{(0)}_k\le\nu^{(0)}_k<\ell_\infty
$$ 
from which it becomes apparent that for $\ell_0=0$ they are never compatible. On the other hand, if we take into account Remark \ref{rmk:remark-eta1}, under the weaker condition \eqref{eq:alternative}, in general we cannot exclude that the intersection of the two conditions in $(H_-)$, when $\ell_0=0$, is empty.
\end{remark}

\section*{Acknowledgments}
\noindent 
The author wishes to thank Professors Vladimir Bobkov and Mieko Tanaka for pointing out a mistake in the original version of the paper and for giving insights into possible solutions, and Professors Rossella Bartolo, Anna Maria Candela, and Addolorata Salvatore for useful discussions and help. 
%
\bibliographystyle{abbrv}
\def\cprime{$'$}

\end{document}